\documentclass[11pt]{amsart}

\usepackage{tikz}

\usetikzlibrary{calc}
\newcommand*\rowst{2}
\newcommand*\rowsa{6}
\newcommand*\rowsb{12}

\numberwithin{equation}{section}

\newtheorem{theorem}{Theorem}[section]
\newtheorem{corollary}[theorem]{Corollary}
\newtheorem{proposition}[theorem]{Proposition}

\newtheorem{conjecture}[theorem]{Conjecture}

\DeclareMathOperator{\lcm}{lcm}
\renewcommand{\a}{\mathbf{a}}

\begin{document}

\title[Unique maximal independent sets]{Unique maximum independent sets in graphs on monomials of a fixed degree}
\author[John Machacek]{John Machacek}
\address{
Department of Mathematics \& Computer Science\\ Hampden-Sydney College\\ Hampden-Sydney, VA 23943\\ USA
}
\email{jmachacek@hsc.edu}


\begin{abstract}
We consider graphs on monomials in $n$ variables of a fixed degree $d$ where two monomials are adjacent if and only if their least common multiple has degree $d+1$.
We prove that when $n = 3$ and $d$ is divisible by $3$ as well as when $n=4$ and $d$ is even that these graphs have a unique maximum independent set.
Domination in these graphs is also considered, and we conjecture that there is equality of the domination number and independent domination number in all cases.
\end{abstract}
\maketitle

\section{Introduction}
An \emph{independent set} (also known as a \emph{stable set}) in a graph is a set of vertices such that there is no edge between any two vertices in the set.
A \emph{maximum independent set} in a graph is an independent set of vertices of maximum possible size.
Finding and understanding maximum independent sets is an important problem from both graph theoretic and algorithmic perspectives.
Our main focus will be on demonstrating certain graphs have a unique maximum independent set.
We will also look at domination and independent domination.
In particular, we show some triangular grid graphs and tetrahedral grid graphs have a unique maximum independent set.
What we call triangular grid graphs and tetrahedral grid graphs are part of a larger family of graphs defined on monomials of a given degree in a polynomial ring.
We also conjecture for these graphs that the domination and independent domination number are equal.
Previous study of independence and domination in these graphs has been done from perspectives of commutative algebra~\cite{GGR1986} and hexagon chess~\cite{chessBound,chess}.

In this introduction we define this family of graphs that include triangular grid graphs and tetrahedral grid graphs.
Section~\ref{sec:triangle} contains our main result of triangular grid graphs, and in Section~\ref{sec:tetra} we prove our main result on tetrahedral grid graphs.
We conclude with Section~\ref{sec:conclusion} which discusses some connections to other work, domination, open problems, and directions for future research.

We now define the family of graphs which will be our focus.
Let $R_n = \mathbb{F}[x_1, x_2, \dots, x_n]$ be the polynomial ring in $n$ variables over the field $\mathbb{F}$.
Let $S_n(d)$ denote the set monomials of degree $d$ in $R_n$.
We then define the graph $G_n(d)$ to be the graph with vertex set $S_n(d)$ which has the edge $\{f,g\}$ for $f,g \in S_n(d)$ if and only if the degree of $\lcm(f,g)$ is $d+1$.
Equivalently, we can consider the graph with vertex set consisting of $n$-tuples of nonnegative integers summing to exactly $d$ where two such $n$-tuples are connected by an edge if and only if they are at Manhattan distance (i.e. $L_1$ distance) of $2$ from each other.
We choose to represent the vertices of these graphs as monomials because there is motivation from commutative algebra to study independent sets in these graphs.
When we take $n = 3$, the graph $G_3(d)$ is called a \emph{triangular grid graph}.
The graph $G_3(3)$ is depicted in Figure~\ref{fig:labeled} on the left with vertices labeled by monomials and on the right with vertices unlabeled.
With larger examples we will draw the graphs $G_n(d)$ without labeling the vertices.
For $n = 4$ we call $G_4(d)$ a \emph{tetrahedral grid graph}.

The \emph{independence number} $\alpha(G)$ of a graph $G$ is the maximum possible size of an independent set.
Thus a maximum independent set of $G$ is precisely an independent set of size $\alpha(G)$.
Let $G$ be a graph with vertex set $V$ and edge set $E$.
For $X \subseteq V$ we use $G|_X$ to denote the graph with vertex set $X$ and edge set $\{\{a,b\} \in E : a,b \in X\}$.
We also use $G \setminus X$ to denote $G|_{\overline{X}}$ where $\overline{X}$ is the complement of $X$ in $V$.

Let $\alpha_n(d) = \alpha(G_n(d))$ denote the independence number of the graph $G_n(d)$.
We will be interested in the quantity $\a_n(d)$ which we define to be the number of maximum independent sets of $G_n(d)$.
The numbers $\alpha_n(d)$ were studied by Geramita, Gregory, and Roberts in the context of monomial ideals~\cite{GGR1986}.
Work on computing the numbers $\alpha_n(d)$ by constructing simplicial complexes has been done by Carlini, H\`a, and Van Tuyl~\cite{EHVT2001}.
A connection between the numbers $\alpha_4(d)$ and sequence A053307 in Sloane's OEIS~\cite{OEIS} which involves certain $2 \times 2$ integer matrices has been found by Babcock and Van Tuyl~\cite{BVT2013}.

The quantity $\alpha_n(d)$ is referred to as the \emph{spreading number} since it records how much a collection of monomials in $S_n(d)$  can ``spread'' in $S_n(d+1)$ when each monomial is multiplied by each variable $x_1, x_2, \dots, x_n$.
That is, $\alpha_n(d)$ is the largest possible size of a subset $A \subseteq S_n(d)$ such that $n|A| = |S_n(1)\cdot A|$ where
\[S_n(1)\cdot A = \{ x_i f : f \in A, 1 \leq 1 \leq n\}\]
which hints at the connection to monomial ideals.
The precise use of the quantity $\alpha_n(d)$ in the work on ideal generation can be found in~\cite[Theorem 4.7]{GGR1986}.
The graphs $G_n(d)$ for $n = 3, 4$ have been further used in commutative algebra~\cite{28pts,Curtis,3and4}.

\begin{figure}
    \centering
    \begin{tikzpicture}[scale=0.75]
    \node (a) at ($2*(1,{sqrt(3)})$) {$x_1^2$};
    \node (b) at ($2*(0.5,{0.5*sqrt(3)})$) {$x_1 x_2$};
    \node (c) at ($2*(1.5,{0.5*sqrt(3)})$) {$x_1 x_3$};
    \node (d) at ($2*(0,0)$) {$x_2^2$};
    \node (e) at ($2*(1,0)$) {$x_2 x_3$};
    \node (f) at ($2*(2,0)$) {$x_3^2$};
    \node (x) at ($2*(-0.5,{-0.5*sqrt(3)})$) {$ x_2^3$};
    \node (y) at ($2*(0.5,{-0.5*sqrt(3)})$) {$x_2^2x_3$};
    \node (z) at ($2*(1.5,{-0.5*sqrt(3)})$) {$x_2 x_3^2$};
    \node (w) at ($2*(2.5,{-0.5*sqrt(3)})$) {$x_3^3$};
    
    \draw (a)--(b)--(c)--(a);
    \draw (b)--(d)--(e)--(b);
    \draw (c)--(f)--(e)--(c);
     \draw (x)--(y)--(z)--(w);
    \draw (x)--(d)--(y)--(e)--(z)--(f)--(w);

    \foreach \row in {0, 1, ...,\rowst} {
        \draw ($\row*(0.5, {0.5*sqrt(3)})$) + (7,0) -- ($(\rowst+7,0)+\row*(-0.5, {0.5*sqrt(3)})$);
        \draw ($\row*(1, 0)$) + (7,0) -- ($(\rowst/2+7,{\rowst/2*sqrt(3)})+\row*(0.5,{-0.5*sqrt(3)})$);
        \draw ($\row*(1, 0)$) + (7,0) -- ($(7,0)+\row*(0.5,{0.5*sqrt(3)})$);
    }

    \end{tikzpicture}
    \caption{The graph $G_3(3)$ both with vertices labeled by monomials and also unlabeled.}
    \label{fig:labeled}
\end{figure}
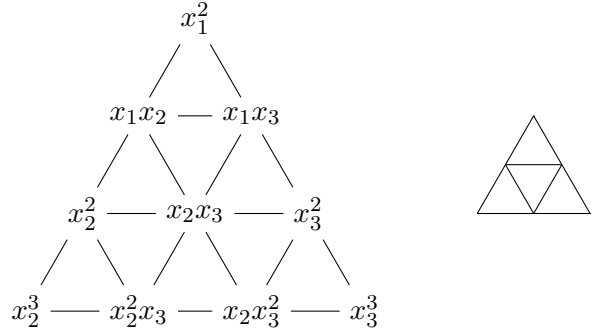

We will study these graphs from a graph theoretic perspective.
Our focus is on constructing maximum independent sets and showing when they are unique.
One motivation for our work is the following conjecture made by A. Howroyd in sequence A297557 of the OEIS~\cite{OEIS}.

\begin{conjecture}[A. Howroyd]
The sequence $\{\a_3(d)\}_{d \geq 9}$ is $3$-periodic with $\a_3(9) = 1$, $\a_3(10) = 27$, and $\a_3(11) = 27$.
That is, for $d \geq 9$ we have 
\[\a_3(d) =
\begin{cases}
1, & \text{if } d \equiv 0 \pmod{3};\\
27, & \text{if } d \not\equiv 0\pmod{3}.
\end{cases}\]
\label{conj:AH}
\end{conjecture}

In Theorem~\ref{thm:triangle} we show that $\a_3(d) = 1$ when $d \geq 9$ and $d \equiv 0 \pmod{3}$.
This means that for such values of $d$ the graph $G_3(d)$ has a unique independent set of maximum size.
This theorem agrees with the more general Conjecture~\ref{conj:AH}.
In Theorem~\ref{thm:4} we show that if $d \equiv 0 \pmod{2}$, then $\a_4(d) = 1$.
We have neither a proven nor conjectural value of $\a_4(d)$ with $d \equiv 1 \pmod{2}$.
Analogous problems for larger values of $n$ are open and discussed in Section~\ref{sec:larger}.

\section{Triangular grid graphs}\label{sec:triangle}
Geramita, Gregory, and Roberts~\cite[Theorem 5.4 (2)]{GGR1986} have computed that
\[\alpha_3(d) = \left\lceil \frac{(d+2)(d+1)}{6} \right\rceil\]
for $d \geq 0$ and $d \not\in \{2,4\}$.
For $k \geq 3$ we see that
\begin{align*}
\alpha_3(3k) &= \left\lceil \frac{3k(3k+1)}{6} + k + \frac{1}{3} \right\rceil \\
\alpha_3(3k-1) &= \left\lceil\frac{3k(3k+1)}{6} \right\rceil\\
\alpha_3(3k-2) &= \left\lceil\frac{3k(3k+1)}{6}- k  \right\rceil\\
\alpha_3(3k-3) &= \left\lceil \frac{3k(3k+1)}{6} - 2k + \frac{1}{3} \right\rceil 
\end{align*}
and observe that
\[ \frac{3k(3k+1)}{6} \in \mathbb{Z} \]
since $3$ divides $3k$ and $2$ divides either $3k$ or $3k + 1$.
Thus,
\begin{align}
\alpha_3(3k) &= \alpha_3(3k-1) + k + 1 \label{eq:3k-1}\\
\alpha_3(3k) &= \alpha_3(3k-2) + 2k + 1 \label{eq:3k-2}\\
\alpha_3(3k) &= \alpha_3(3k-3) + 3k\label{eq:3k-3}
\end{align}
for $k \geq 3$.

\begin{theorem}
If $d \neq 6$ and $d \equiv 0 \pmod{3}$, then $\a_3(d) = 1$.
\label{thm:triangle}
\end{theorem}

\begin{proof}
The fact that $\a_3(0) = 1$, $\a_3(3)=1$, and $\a_3(9) = 1$ can by verified by direct computation.
Now take some $k > 3$, let $d = 3k$ and consider $G=G_3(d)$.
We define
\[X = \{x_2^ix_3^{d-i} : 0 \leq i \leq d\}\]
and
\[Y = \{x_1x_2^ix_3^{d-i-1} : 0 \leq i < d\}.\]
We let $G' = G \setminus X$ and $G'' = G' \setminus Y$.
We note that $G' \cong G_3(d-1)$ and $G'' \cong G_3(d-2)$.

Take a maximum size independent set $I$ in $G$ which must have $|I| = \alpha_3(d)$.
The size of $I$ restricted to $G'$ and $G''$ can be at most $\alpha_3(d-1)$ and $\alpha_3(d-2)$ respectively.
By Equations~(\ref{eq:3k-1}) and~(\ref{eq:3k-2}) it follows that $|I \cap X| \geq k+1$ and $|I \cap (X \cup Y)| \geq 2k+1$.
Observe that $|X| = 3k+1$ and $|Y| = 3k$.

Let $s = |I \cap X|$ and $t = |I \cap Y| $.
So, $s \geq k+1$ and $s + t \geq 2k+1$.
Every vertex of $X$ is adjacent to two vertices in $Y$ with the exception of $x_2^d$ and $x_3^d$ that are each adjacent to a single vertex in $Y$.
Furthermore, any two distinct vertices of $I \cap X$ will have no common neighbors in $Y$ because $I$ is an independent set.
It follows that $t \leq 3k - 2s +2$.
We then see that
\[2k+1 - s \leq 3k-2s+2\]
which implies $s \leq k+1$.
Thus we can conclude that $s = k+1$ and $t = k$.

Each vertex in $Y$ is in bijective correspondence with the edge connecting its two neighboring vertices which are in $X$. 
Since $|I \cap Y| = k$ and $|I \cap X| = k+1$, this means we need a matching of size $k$ inside $G|_X$ along with an independent set of size $k+1$ in $X$ disjoint from the matching.
This requires every vertex of $X$ to be covered either by an edge in the matching or to be vertex in the independent set.
We see that the only way to do this to take the vertices
\[\{x_2^{d-3i}x_3^{3i} : 0 \leq i \leq k\}\]
as the independent set along with the edges
\[\{\{x_2^{d-3i-1}x_3^{3i+1}, x_2^{d-3i-2}x_3^{3i+2}\} : 0 \leq i < k\}\]
as the matching.
It follows that $\{x_2^{d-3i}x_3^{3i} : 0 \leq i \leq k\} \subseteq I$.
An example of this configuration can be seen in the bottom two rows of the triangular grid graph in Figure~\ref{fig:G_3(12)}.

We could have chosen $X$ to be the vertices where the exponent of $x_i$ was zero for any $1 \leq i \leq 3$ and the same argument applies. It follows $J \subseteq I$ where
\[J = \{x_1^{d-3i}x_2^{3i} : 0 \leq i \leq k\} \cup \{x_1^{d-3i}x_3^{3i} : 0 \leq i \leq k\} \cup \{x_2^{d-3i}x_3^{3i} : 0 \leq i \leq k\}.\]
We see that $|J| = 3k$.
Consider the graph $G''' = G \setminus J$ and observe that $G''' \cong G_3(d-3)$.
Now we have
\[\alpha(G''') = \alpha_3(d-3) = \alpha(G) - 3k\]
by Equation~(\ref{eq:3k-3}).
By induction $G'''$ has a unique maximum independent set of size $\alpha_3(d-3)$.
Therefore, $I$ is the unique maximum independent set of $G$ and the theorem is proven. 
\end{proof}

In Figure~\ref{fig:six} we show the two possible maximum independent sets of $S_3(6)$ demonstrating that Theorem~\ref{thm:triangle} does not hold for $d=6$.
The unique maximum independent set of $G_3(12)$ is shown in Figure~\ref{fig:G_3(12)}, and from this illustration one can see the general pattern of the unique maximum independent set of $G_3(3k)$ for $k \geq 3$.
Theorem~\ref{thm:triangle} verifies Conjecture~\ref{conj:AH} for $d \equiv 0 \pmod{3}$.

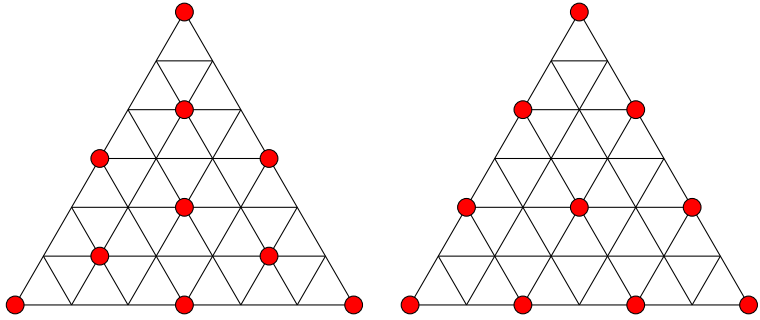
\begin{figure}
    \centering
    \begin{tikzpicture}[scale=0.75]
    \foreach \row in {0, 1, ...,\rowsa} {
        \draw ($\row*(0.5, {0.5*sqrt(3)})$) -- ($(\rowsa,0)+\row*(-0.5, {0.5*sqrt(3)})$);
        \draw ($\row*(1, 0)$) -- ($(\rowsa/2,{\rowsa/2*sqrt(3)})+\row*(0.5,{-0.5*sqrt(3)})$);
        \draw ($\row*(1, 0)$) -- ($(0,0)+\row*(0.5,{0.5*sqrt(3)})$);
    }
    
    \foreach \row in {0, 1, ...,\rowsa} {
        \draw ($\row*(0.5, {0.5*sqrt(3)})$) + (7,0) -- ($(\rowsa+7,0)+\row*(-0.5, {0.5*sqrt(3)})$);
        \draw ($\row*(1, 0)$) + (7,0) -- ($(\rowsa/2+7,{\rowsa/2*sqrt(3)})+\row*(0.5,{-0.5*sqrt(3)})$);
        \draw ($\row*(1, 0)$) + (7,0) -- ($(7,0)+\row*(0.5,{0.5*sqrt(3)})$);
    }
    
    \node[draw, circle, fill=red!, scale=0.65] at (0,0) {};
    \node[draw, circle, fill=red!, scale=0.65] at (3,0) {};
    \node[draw, circle, fill=red!, scale=0.65] at (6,0) {};
    
    \node[draw, circle, fill=red!, scale=0.65] at (1.5,{0.5*sqrt(3)}) {};
    \node[draw, circle, fill=red!, scale=0.65] at (4.5,{0.5*sqrt(3)}) {};

    \node[draw, circle, fill=red!, scale=0.65] at (3,{sqrt(3)}) {};

    \node[draw, circle, fill=red!, scale=0.65] at (1.5,{1.5*sqrt(3)}) {};
    \node[draw, circle, fill=red!, scale=0.65] at (4.5,{1.5*sqrt(3)}) {};
    
    \node[draw, circle, fill=red!, scale=0.65] at (3,{2*sqrt(3)}) {};
    
    \node[draw, circle, fill=red!, scale=0.65] at (3,{3*sqrt(3)}) {};
    
    \node[draw, circle, fill=red!, scale=0.65] at (7,0) {};
    \node[draw, circle, fill=red!, scale=0.65] at (9,0) {};
    \node[draw, circle, fill=red!, scale=0.65] at (11,0) {};
    \node[draw, circle, fill=red!, scale=0.65] at (13,0) {};
    
    \node[draw, circle, fill=red!, scale=0.65] at (8,{sqrt(3)}) {};
    \node[draw, circle, fill=red!, scale=0.65] at (10,{sqrt(3)}) {};
    \node[draw, circle, fill=red!, scale=0.65] at (12,{sqrt(3)}) {};
    
    \node[draw, circle, fill=red!, scale=0.65] at (9,{2*sqrt(3)}) {};
    \node[draw, circle, fill=red!, scale=0.65] at (11,{2*sqrt(3)}) {};
    
    \node[draw, circle, fill=red!, scale=0.65] at (10,{3*sqrt(3)}) {};
    \end{tikzpicture}
    \caption{The two maximum independent sets of $G_3(6)$ of size $\alpha_3(6) = 10$.}
    \label{fig:six}
\end{figure}

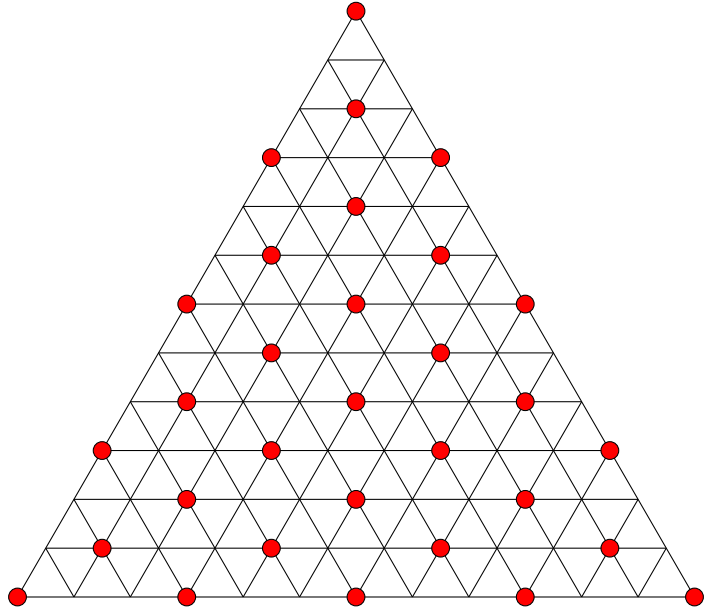
\begin{figure}
    \centering
    \begin{tikzpicture}[scale=0.75]
    \foreach \row in {0, 1, ...,\rowsb} {
        \draw ($\row*(0.5, {0.5*sqrt(3)})$) -- ($(\rowsb,0)+\row*(-0.5, {0.5*sqrt(3)})$);
        \draw ($\row*(1, 0)$) -- ($(\rowsb/2,{\rowsb/2*sqrt(3)})+\row*(0.5,{-0.5*sqrt(3)})$);
        \draw ($\row*(1, 0)$) -- ($(0,0)+\row*(0.5,{0.5*sqrt(3)})$);
    }
    \node[draw, circle, fill=red!, scale=0.65] at (0,0) {};
    \node[draw, circle, fill=red!, scale=0.65] at (3,0) {};
    \node[draw, circle, fill=red!, scale=0.65] at (6,0) {};
    \node[draw, circle, fill=red!, scale=0.65] at (9,0) {};
    \node[draw, circle, fill=red!, scale=0.65] at (12,0) {};
    
    \node[draw, circle, fill=red!, scale=0.65] at (1.5,{0.5*sqrt(3)}) {};
    \node[draw, circle, fill=red!, scale=0.65] at (4.5,{0.5*sqrt(3)}) {};
    \node[draw, circle, fill=red!, scale=0.65] at (7.5,{0.5*sqrt(3)}) {};
    \node[draw, circle, fill=red!, scale=0.65] at (10.5,{0.5*sqrt(3)}) {};
    
    \node[draw, circle, fill=red!, scale=0.65] at (3,{sqrt(3)}) {};
    \node[draw, circle, fill=red!, scale=0.65] at (6,{sqrt(3)}) {};
    \node[draw, circle, fill=red!, scale=0.65] at (9,{sqrt(3)}) {};
    
    \node[draw, circle, fill=red!, scale=0.65] at (1.5,{1.5*sqrt(3)}) {};
    \node[draw, circle, fill=red!, scale=0.65] at (4.5,{1.5*sqrt(3)}) {};
    \node[draw, circle, fill=red!, scale=0.65] at (7.5,{1.5*sqrt(3)}) {};
    \node[draw, circle, fill=red!, scale=0.65] at (10.5,{1.5*sqrt(3)}) {};
    
    \node[draw, circle, fill=red!, scale=0.65] at (3,{2*sqrt(3)}) {};
    \node[draw, circle, fill=red!, scale=0.65] at (6,{2*sqrt(3)}) {};
    \node[draw, circle, fill=red!, scale=0.65] at (9,{2*sqrt(3)}) {};
    
    \node[draw, circle, fill=red!, scale=0.65] at (4.5,{2.5*sqrt(3)}) {};
    \node[draw, circle, fill=red!, scale=0.65] at (7.5,{2.5*sqrt(3)}) {};
    
    \node[draw, circle, fill=red!, scale=0.65] at (3,{3*sqrt(3)}) {};
    \node[draw, circle, fill=red!, scale=0.65] at (6,{3*sqrt(3)}) {};
    \node[draw, circle, fill=red!, scale=0.65] at (9,{3*sqrt(3)}) {};
    
    \node[draw, circle, fill=red!, scale=0.65] at (4.5,{3.5*sqrt(3)}) {};
    \node[draw, circle, fill=red!, scale=0.65] at (7.5,{3.5*sqrt(3)}) {};
    
    \node[draw, circle, fill=red!, scale=0.65] at (6,{4*sqrt(3)}) {};
    
    \node[draw, circle, fill=red!, scale=0.65] at (4.5,{4.5*sqrt(3)}) {};
    \node[draw, circle, fill=red!, scale=0.65] at (7.5,{4.5*sqrt(3)}) {};
    
    \node[draw, circle, fill=red!, scale=0.65] at (6,{5*sqrt(3)}) {};
    
    \node[draw, circle, fill=red!, scale=0.65] at (6,{6*sqrt(3)}) {};
    \end{tikzpicture}
    \caption{The unique maximum independent set of $G_3(12)$ of size $\alpha_3(12) = 31$.}
    \label{fig:G_3(12)}
\end{figure}

\section{The tetrahedral grid graph}\label{sec:tetra}
Geramita, Gregory, and Roberts~\cite[Proposition 5.6]{GGR1986} have computed that
\[\alpha_4(2k+1) = \frac{(k+2)(2k+3)(k+1)}{6} \]
and also~\cite[Remarks 5.7 (ii)]{GGR1986} that
\[\alpha_4(2k) = \frac{(k+1)^3 + 2(k+1)}{3}\]
for any $k \geq 0$.
This means that
\begin{align}
    \alpha_4(2k) - \alpha_4(2k-1) &= \frac{(k+2)(k+1)}{2}\label{eq:41}\\
    \alpha_4(2k) - \alpha_4(2k-2) &= k^2 + k + 1\label{eq:42}
\end{align}
We will make use of these differences in the proof of Theorem~\ref{thm:4}.
Let us also record two computations which will be of use to us.
First let 
\[\epsilon_m = \begin{cases} 1 & m \equiv 0 \pmod{3} \\ 0 & \text{otherwise} \end{cases}\]
for any integer $m$.
Then it is the case that
\begin{equation}
\epsilon_m + \sum_{i=0}^{\lfloor \frac{m}{3} \rfloor} (3m - 9i) = \frac{(m+2)(m+1)}{2}
\label{eq:sum0}
\end{equation}
and
\begin{equation}
\epsilon_m + \sum_{i=1}^{\lfloor \frac{m}{3} \rfloor} (3m - 9i) = \frac{(m-2)(m-1)}{2}
\label{eq:sum1}
\end{equation}
both of which can verified considering cases $m \pmod{3}$.

\begin{figure}
    \centering
\begin{tikzpicture}[scale=0.75]
\node[draw, circle, fill=red!, scale=0.65] at (0,0) {};

\begin{scope}[shift={(1,0)}]
\newcommand*\rows{1}
\foreach \row in {0, 1, ...,\rows} {
        \draw ($\row*(0.5, {0.5*sqrt(3)})$) -- ($(\rows,0)+\row*(-0.5, {0.5*sqrt(3)})$);
        \draw ($\row*(1, 0)$) -- ($(\rows/2,{\rows/2*sqrt(3)})+\row*(0.5,{-0.5*sqrt(3)})$);
        \draw ($\row*(1, 0)$) -- ($(0,0)+\row*(0.5,{0.5*sqrt(3)})$);
    }
\draw[fill=gray!50] (1,0) -- (0.5,{0.5*sqrt(3)}) -- (0, 0) -- (1,0);
\end{scope}

\begin{scope}[shift={(3,0)}]
\newcommand*\rows{2}
\foreach \row in {0, 1, ...,\rows} {
        \draw ($\row*(0.5, {0.5*sqrt(3)})$) -- ($(\rows,0)+\row*(-0.5, {0.5*sqrt(3)})$);
        \draw ($\row*(1, 0)$) -- ($(\rows/2,{\rows/2*sqrt(3)})+\row*(0.5,{-0.5*sqrt(3)})$);
        \draw ($\row*(1, 0)$) -- ($(0,0)+\row*(0.5,{0.5*sqrt(3)})$);
    }
\node[draw, circle, fill=red!, scale=0.65] at (0,0) {};
\node[draw, circle, fill=red!, scale=0.65] at (2,0) {};
\node[draw, circle, fill=red!, scale=0.65] at (1,{sqrt(3)}) {};
\end{scope}

\begin{scope}[shift={(6,0)}]
\newcommand*\rows{3}
\foreach \row in {0, 1, ...,\rows} {
        \draw ($\row*(0.5, {0.5*sqrt(3)})$) -- ($(\rows,0)+\row*(-0.5, {0.5*sqrt(3)})$);
        \draw ($\row*(1, 0)$) -- ($(\rows/2,{\rows/2*sqrt(3)})+\row*(0.5,{-0.5*sqrt(3)})$);
        \draw ($\row*(1, 0)$) -- ($(0,0)+\row*(0.5,{0.5*sqrt(3)})$);
    }
   \draw[fill=gray!50] (1,0) -- (0.5,{0.5*sqrt(3)}) -- (0, 0) -- (1,0);
   \draw[fill=gray!50] (3,0) -- (2.5,{0.5*sqrt(3)}) -- (2, 0) -- (3,0);
   \draw[fill=gray!50] (2,{sqrt(3)}) -- (1.5,{1.5*sqrt(3)}) -- (1, {sqrt(3)}) -- (2,{sqrt(3)});
   \node[draw, circle, fill=red!, scale=0.65] at (1.5,{0.5*sqrt(3)}) {};
\end{scope}
    
\begin{scope}[shift={(10,0)}]
\newcommand*\rows{4}
\foreach \row in {0, 1, ...,\rows} {
        \draw ($\row*(0.5, {0.5*sqrt(3)})$) -- ($(\rows,0)+\row*(-0.5, {0.5*sqrt(3)})$);
        \draw ($\row*(1, 0)$) -- ($(\rows/2,{\rows/2*sqrt(3)})+\row*(0.5,{-0.5*sqrt(3)})$);
        \draw ($\row*(1, 0)$) -- ($(0,0)+\row*(0.5,{0.5*sqrt(3)})$);
    }
   \node[draw, circle, fill=red!, scale=0.65] at (0,0) {};
   \node[draw, circle, fill=red!, scale=0.65] at (2,0) {};
   \node[draw, circle, fill=red!, scale=0.65] at (4,0) {};
   \node[draw, circle, fill=red!, scale=0.65] at (1,{sqrt(3)}) {};
   \node[draw, circle, fill=red!, scale=0.65] at (3,{sqrt(3)}) {};
   \node[draw, circle, fill=red!, scale=0.65] at (2,{2*sqrt(3)}) {};
   \draw[fill=gray!50] (1.5,{0.5*sqrt(3)}) -- (2.5,{0.5*sqrt(3)}) -- (2, {sqrt(3)}) -- (1.5,{0.5*sqrt(3)});
\end{scope}
  
\begin{scope}[shift={(0,-6)}]
\newcommand*\rows{5}
\foreach \row in {0, 1, ...,\rows} {
        \draw ($\row*(0.5, {0.5*sqrt(3)})$) -- ($(\rows,0)+\row*(-0.5, {0.5*sqrt(3)})$);
        \draw ($\row*(1, 0)$) -- ($(\rows/2,{\rows/2*sqrt(3)})+\row*(0.5,{-0.5*sqrt(3)})$);
        \draw ($\row*(1, 0)$) -- ($(0,0)+\row*(0.5,{0.5*sqrt(3)})$);
    }
       \draw[fill=gray!50] (0,0) -- (1,0) -- (0.5, {0.5*sqrt(3)}) -- (0,0);
       \draw[fill=gray!50] (2,0) -- (3,0) -- (2.5, {0.5*sqrt(3)}) -- (2,0);
       \draw[fill=gray!50] (4,0) -- (5,0) -- (4.5, {0.5*sqrt(3)}) -- (4,0);
       \draw[fill=gray!50] (1,{sqrt(3)}) -- (2,{sqrt(3)}) -- (1.5, {1.5*sqrt(3)}) -- (1,{sqrt(3)});
       \draw[fill=gray!50] (3,{sqrt(3)}) -- (4,{sqrt(3)}) -- (3.5, {1.5*sqrt(3)}) -- (3,{sqrt(3)});
       \draw[fill=gray!50] (2,{2*sqrt(3)}) -- (3,{2*sqrt(3)}) -- (2.5, {2.5*sqrt(3)}) -- (2,{2*sqrt(3)});
        \node[draw, circle, fill=red!, scale=0.65] at (1.5,{0.5*sqrt(3)}) {};
        \node[draw, circle, fill=red!, scale=0.65] at (3.5,{0.5*sqrt(3)}) {};
         \node[draw, circle, fill=red!, scale=0.65] at (2.5,{1.5*sqrt(3)}) {};
 \end{scope}

\begin{scope}[shift={(6,-6)}]
\newcommand*\rows{6}
\foreach \row in {0, 1, ...,\rows} {
        \draw ($\row*(0.5, {0.5*sqrt(3)})$) -- ($(\rows,0)+\row*(-0.5, {0.5*sqrt(3)})$);
        \draw ($\row*(1, 0)$) -- ($(\rows/2,{\rows/2*sqrt(3)})+\row*(0.5,{-0.5*sqrt(3)})$);
        \draw ($\row*(1, 0)$) -- ($(0,0)+\row*(0.5,{0.5*sqrt(3)})$);
    }
        \node[draw, circle, fill=red!, scale=0.65] at (0,0) {};
        \node[draw, circle, fill=red!, scale=0.65] at (2,0) {};
        \node[draw, circle, fill=red!, scale=0.65] at (4,0) {};
        \node[draw, circle, fill=red!, scale=0.65] at (6,0) {};
        \node[draw, circle, fill=red!, scale=0.65] at (1,{sqrt(3)}) {};
        \node[draw, circle, fill=red!, scale=0.65] at (3,{sqrt(3)}) {};
        \node[draw, circle, fill=red!, scale=0.65] at (5,{sqrt(3)}) {};
        \node[draw, circle, fill=red!, scale=0.65] at (2,{2*sqrt(3)}) {};
        \node[draw, circle, fill=red!, scale=0.65] at (4,{2*sqrt(3)}) {};
         \node[draw, circle, fill=red!, scale=0.65] at (3,{3*sqrt(3)}) {};
         \draw[fill=gray!50] (1.5,{0.5*sqrt(3)}) -- (2.5,{0.5*sqrt(3)}) -- (2, {sqrt(3)}) -- (1.5,{0.5*sqrt(3)});
         \draw[fill=gray!50] (3.5,{0.5*sqrt(3)}) -- (4.5,{0.5*sqrt(3)}) -- (4, {sqrt(3)}) -- (3.5,{0.5*sqrt(3)});
         \draw[fill=gray!50] (2.5,{1.5*sqrt(3)}) -- (3.5,{1.5*sqrt(3)}) -- (3, {2*sqrt(3)}) -- (2.5,{1.5*sqrt(3)});
\end{scope}

\begin{scope}[shift={(-2,-14)}]
\newcommand*\rows{7}
\foreach \row in {0, 1, ...,\rows} {
        \draw ($\row*(0.5, {0.5*sqrt(3)})$) -- ($(\rows,0)+\row*(-0.5, {0.5*sqrt(3)})$);
        \draw ($\row*(1, 0)$) -- ($(\rows/2,{\rows/2*sqrt(3)})+\row*(0.5,{-0.5*sqrt(3)})$);
        \draw ($\row*(1, 0)$) -- ($(0,0)+\row*(0.5,{0.5*sqrt(3)})$);
    }
         \draw[fill=gray!50] (0,0) -- (1,0) -- (0.5, {0.5*sqrt(3)}) -- (0,0);
         \draw[fill=gray!50] (2,0) -- (3,0) -- (2.5, {0.5*sqrt(3)}) -- (2,0);
         \draw[fill=gray!50] (4,0) -- (5,0) -- (4.5, {0.5*sqrt(3)}) -- (4,0);
         \draw[fill=gray!50] (6,0) -- (7,0) -- (6.5, {0.5*sqrt(3)}) -- (6,0);
         \draw[fill=gray!50] (1,{sqrt(3)}) -- (2,{sqrt(3)}) -- (1.5, {1.5*sqrt(3)}) -- (1,{sqrt(3)});
         \draw[fill=gray!50] (3,{sqrt(3)}) -- (4,{sqrt(3)}) -- (3.5, {1.5*sqrt(3)}) -- (3,{sqrt(3)});
         \draw[fill=gray!50] (5,{sqrt(3)}) -- (6,{sqrt(3)}) -- (5.5, {1.5*sqrt(3)}) -- (5,{sqrt(3)});
          \draw[fill=gray!50] (2,{2*sqrt(3)}) -- (3,{2*sqrt(3)}) -- (2.5, {2.5*sqrt(3)}) -- (2,{2*sqrt(3)});
          \draw[fill=gray!50] (4,{2*sqrt(3)}) -- (5,{2*sqrt(3)}) -- (4.5, {2.5*sqrt(3)}) -- (4,{2*sqrt(3)});
          \draw[fill=gray!50] (3,{3*sqrt(3)}) -- (4,{3*sqrt(3)}) -- (3.5, {3.5*sqrt(3)}) -- (3,{3*sqrt(3)});
          \node[draw, circle, fill=red!, scale=0.65] at (1.5,{0.5*sqrt(3)}) {};
          \node[draw, circle, fill=red!, scale=0.65] at (3.5,{0.5*sqrt(3)}) {};
          \node[draw, circle, fill=red!, scale=0.65] at (5.5,{0.5*sqrt(3)}) {};
          \node[draw, circle, fill=red!, scale=0.65] at (2.5,{1.5*sqrt(3)}) {};
          \node[draw, circle, fill=red!, scale=0.65] at (4.5,{1.5*sqrt(3)}) {};
          \node[draw, circle, fill=red!, scale=0.65] at (3.5,{2.5*sqrt(3)}) {};
\end{scope}

\begin{scope}[shift={(6,-14)}]
\newcommand*\rows{8}
\foreach \row in {0, 1, ...,\rows} {
        \draw ($\row*(0.5, {0.5*sqrt(3)})$) -- ($(\rows,0)+\row*(-0.5, {0.5*sqrt(3)})$);
        \draw ($\row*(1, 0)$) -- ($(\rows/2,{\rows/2*sqrt(3)})+\row*(0.5,{-0.5*sqrt(3)})$);
        \draw ($\row*(1, 0)$) -- ($(0,0)+\row*(0.5,{0.5*sqrt(3)})$);
    }
            \node[draw, circle, fill=red!, scale=0.65] at (0,0) {};
        \node[draw, circle, fill=red!, scale=0.65] at (2,0) {};
        \node[draw, circle, fill=red!, scale=0.65] at (4,0) {};
        \node[draw, circle, fill=red!, scale=0.65] at (6,0) {};
        \node[draw, circle, fill=red!, scale=0.65] at (8,0) {};
        \node[draw, circle, fill=red!, scale=0.65] at (1,{sqrt(3)}) {};
        \node[draw, circle, fill=red!, scale=0.65] at (3,{sqrt(3)}) {};
        \node[draw, circle, fill=red!, scale=0.65] at (5,{sqrt(3)}) {};
        \node[draw, circle, fill=red!, scale=0.65] at (7,{sqrt(3)}) {};
        \node[draw, circle, fill=red!, scale=0.65] at (2,{2*sqrt(3)}) {};
        \node[draw, circle, fill=red!, scale=0.65] at (4,{2*sqrt(3)}) {};
        \node[draw, circle, fill=red!, scale=0.65] at (6,{2*sqrt(3)}) {};
         \node[draw, circle, fill=red!, scale=0.65] at (3,{3*sqrt(3)}) {};
         \node[draw, circle, fill=red!, scale=0.65] at (5,{3*sqrt(3)}) {};
         \node[draw, circle, fill=red!, scale=0.65] at (4,{4*sqrt(3)}) {};
         \draw[fill=gray!50] (1.5,{0.5*sqrt(3)}) -- (2.5,{0.5*sqrt(3)}) -- (2, {sqrt(3)}) -- (1.5,{0.5*sqrt(3)});
         \draw[fill=gray!50] (3.5,{0.5*sqrt(3)}) -- (4.5,{0.5*sqrt(3)}) -- (4, {sqrt(3)}) -- (3.5,{0.5*sqrt(3)});
         \draw[fill=gray!50] (5.5,{0.5*sqrt(3)}) -- (6.5,{0.5*sqrt(3)}) -- (6, {sqrt(3)}) -- (5.5,{0.5*sqrt(3)});
         \draw[fill=gray!50] (2.5,{1.5*sqrt(3)}) -- (3.5,{1.5*sqrt(3)}) -- (3, {2*sqrt(3)}) -- (2.5,{1.5*sqrt(3)});
        \draw[fill=gray!50] (4.5,{1.5*sqrt(3)}) -- (5.5,{1.5*sqrt(3)}) -- (5, {2*sqrt(3)}) -- (4.5,{1.5*sqrt(3)});
        \draw[fill=gray!50] (3.5,{2.5*sqrt(3)}) -- (4.5,{2.5*sqrt(3)}) -- (4, {3*sqrt(3)}) -- (3.5,{2.5*sqrt(3)});
\end{scope}

\end{tikzpicture}
\caption{The unique maximal independent set $G_4(8)$ of size $\alpha_4(8) = 45$ shown partitioned into slices.}
\label{fig:slices}
\end{figure}
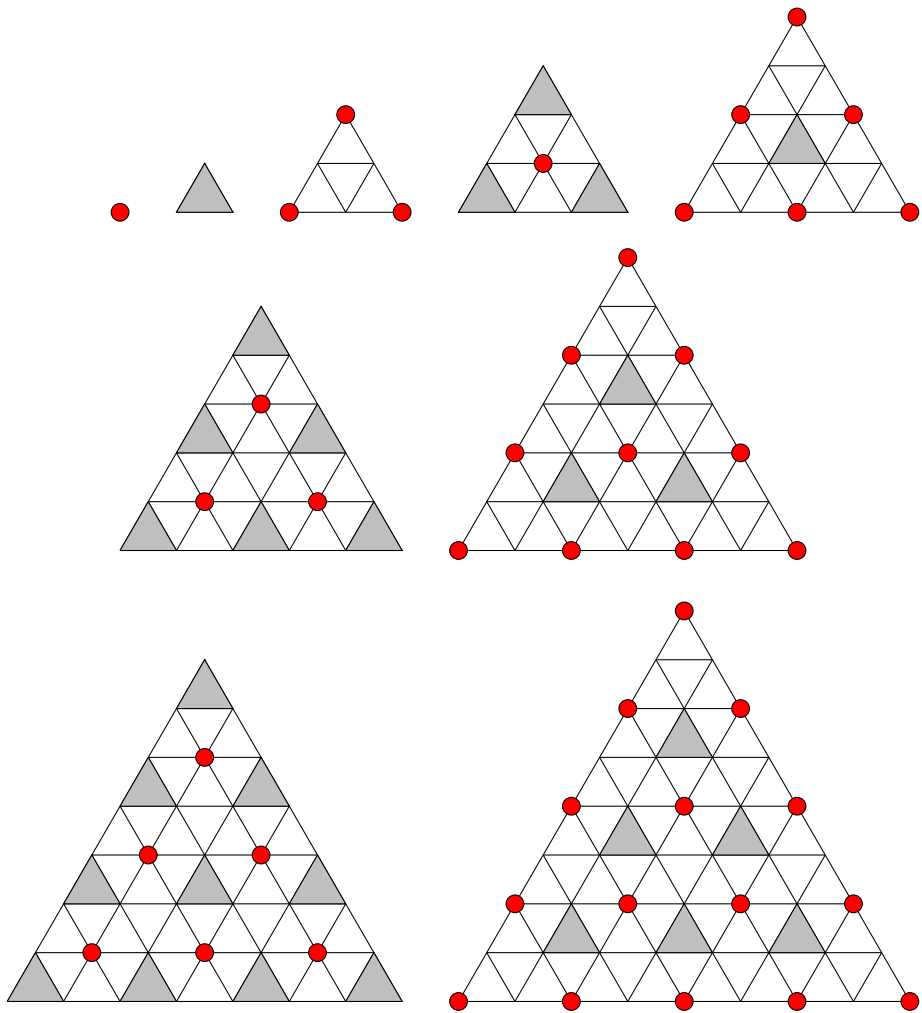

\begin{theorem}
If $d \equiv 0 \pmod{2}$, then $\a_4(d) = 1$.
\label{thm:4}
\end{theorem}
\begin{proof}
It is clear that $\a_4(0) = 1$.
Let $G = G_4(d)$ where $d = 2k$ for some $k > 0$.
Now let $X$ be the set of vertices where  the power of $x_1$ is zero, and let $Y$ be the set of vertices where the power of $x_1$ is one.
We then let $G' = G \setminus X$ and $G'' = G \setminus Y$.
It is then case that $G' \cong G_4(d-1)$ and $G'' \cong G_4(d-2)$.
Take a maximum size independent set $I$ in $G$ which will have $|I| = \alpha_4(d)$.
The size of $I$ restricted to $G'$ and $G''$ can be at most $\alpha_4(d-1)$ and $\alpha_4(d-2)$ respectively.
We let $s = |I \cap X|$ and $t = |I \cap Y|$.
We have that
\begin{align*}
    |X| &= \frac{(2k+2)(2k+1)}{2}\\
    |Y| &= \frac{(2k+1)(2k)}{2}
\end{align*}
while $G|_X \cong G_3(d)$ and $G|_Y \cong G_3(d-1)$.

Using Equations~(\ref{eq:41}) and~(\ref{eq:42}) we find that
\begin{align*}
    s &\geq \frac{(k+2)(k+1)}{2}\\
    s+t &\geq k^2 + k + 1
\end{align*}
from which it follows 
\begin{equation}
t \geq k^2 + k + 1 - s.
\label{eq:A}
\end{equation}
We also have that 
\begin{equation}
t \leq |Y| - 3s + 3k + 3 = 2k^2 + 4k + 3 - 3s
\label{eq:B}
\end{equation}
by considering the adjacency of vertices in $X$ with vertices in $Y$.
Let us explain this bound.
Since $I$ is an independent set any two distinct vertices in $X \cap I$ will have no common neighbors in $Y$.
Consider a vertex $x_2^{e_2} x_3^{e_3} x_4^{e_4} \in X$.
If $e_2, e_3, e_4 > 0$, then this vertex has three neighbors in $Y$.
If exactly one of $e_2$, $e_3$, or $e_4$ is zero, then this vertex has two neighbors in $Y$.
If exactly one of $e_2$, $e_3$, or $e_4$ is nonzero, then this vertex has only one neighbor in $Y$.
The set of vertices in $X$ with at least one of $e_2$, $e_3$ or $e_4$ equal to zero is a cycle of size $3d = 6k$.
Assume $X \cap I$ contains $3k$ vertices $x_2^{e_2} x_3^{e_3} x_4^{e_4}$ with at least one of $e_2$, $e_3$, or $e_4$ equal zero including each of $x_2^d$, $x_3^d$, and $x_4^d$.
In this case the $s$ vertices of $X \cap I$ in total have $3s - 3k - 3$ neighbors in $Y$.
Furthermore, this is the fewest possible neighbors in $Y$ we can have since $I$ is an independent set.
Thus the bound claimed bound on $t$ follows.

Using Equations~(\ref{eq:A}) and~(\ref{eq:B}) we find that
\[s = \frac{(k+2)(k+1)}{2}\]
since we obtain a upper bound on $s$ matching our lower bound.
We can then conclude that
\[t = \frac{k(k-1)}{2}\]
and we have found a requirement on the sizes of both $X \cap I$ and $Y \cap I$.
It remains to show the required sizes of $X \cap I$ and $Y \cap I$ force a unique maximum independent set.
We will now argue that the only possible maximum independent set follows the pattern in Figure~\ref{fig:slices} which shows that maximum independent set in ``slices'' depending of the exponent of $x_1$.

For any $a$ the graph $G_3(a)$ can be partitioned into a disjoint union of cycles of lengths $\{3a - 9i : 0 \leq i \leq \lfloor \frac{a}{3} \rfloor\}$ where we have an isolated vertex for the cycle of length $0$ in the case the $3$ divides $a$.
Here the cycle of length $3a - 9i$ or the isolated vertex is the induced subgraph on the vertices 
\[\{x_2^{e_2}x_3^{e_3}x_4^{e_4} : e_j \geq i \text{ for all } 2 \leq j \leq 4 \text{ and } e_j = i \text{ for some } 2 \leq j \leq 4\}.\]
We have that $G|_X \cong G_3(d)$ and $G|_Y \cong G_3(d-1)$ where $d = 2k$.
We claim the only option for $X \cap I$ is to take an independent set of size
\[\frac{3(2k) - 9(2i)}{2} = 3k - 9i\]
in the cycle of length $3d - 9(2i)$ along with the additional vertex if $3$ divides $d$.
Moreover, these independent sets in the cycles must include $x_2^d x_3^i x_4^i$ which determines the independent set uniquely.
Similarly, we claim the only option for $Y \cap I$ is to take an independent set of size
\[\frac{3(2k-1) - 9(2i+1)}{2} = 3(k+1) - 9(i+1)\]
in the cycle of length $3(d-1) - 9(2i+1)$ along with the additional vertex if $3$ divides $d-1$.
Moreover, these independent sets in the cycles must include $x_2^{(d-1)-(i+1)}x_3^{i+1}x_4^{i+1}$ which determines the independent set uniquely.
Equations~(\ref{eq:sum0})~and~(\ref{eq:sum1}) that the count of vertices in each of $X \cap I$ and $Y \cap I$ is correct.

Since we can partition $G_4(d)$ into ``slices'' isomorphic to $G_3(0), \dots, G_3(d)$ where $G_3(a)$ consists of vertices where the power of $x_1$ is $d-a$ we can see these claimed descriptions of $X \cap I$ and $Y \cap I$ by induction.
For the base case we have the $G_4(0)$ is a single vertex.
So, assume $d = 2k \geq 2$.
Similar to our definition of $X$ and $Y$, let $Z$ be the subset of vertices so that $G_4(d)|_Z \cong G_3(d-2)$.
By induction $Z \cap I$ has the form we claimed above.
This implies $Y \cap I$ has the desired form since $|Y \cap I| = \frac{k(k-1)}{2}$ and the configuration of vertices in $Z \cap I$ restricts the possibility for $Y \cap I$.
In the same way $X \cap I$ is as we claimed since $|X \cap I| = \frac{(k+2)(k+1)}{2}$ and the configuration of $Y \cap I$ is known.

Figure~\ref{fig:slices} shows the unique independent set constructed for $G_4(8)$ partitioned into slices.
The vertices in red are part of the independent set and the triangles shaded in gray are adjacent to vertices from another slice in the independent set.
\end{proof}

\section{Concluding remarks}\label{sec:conclusion}

\subsection{Periodicity and larger values of $n$} 
\label{sec:larger}
We now give two conjectures in the spirit of Conjecture~\ref{conj:AH} , Theorem~\ref{thm:triangle}, and Theorem~\ref{thm:4}.
The first conjecture deals with unique maximum independent sets.
The second conjecture deals with periodicity of the sequences $\{\a_n(d)\}_{d \geq 0}$.

\begin{conjecture}
For any $n \geq 1$ there exists some $N \geq 0$ such that $\a_n(d) = 1$ whenever $d \equiv 0 \pmod{n}$ and $d \geq N$.
\label{conj:1}
\end{conjecture}

\begin{conjecture}
For any $n \geq 1$ there exists some $N \geq 0$ such that the sequence $\{\a_n(d)\}_{d \geq N}$ is $n$-periodic.
\label{conj:period}
\end{conjecture}

It is easy to see Conjecture~\ref{conj:1} and Conjecture~\ref{conj:period} both hold when $n=1$ and $n=2$.
Since $G_1(d)$ is always just a single point we find that  $\a_1(d) = 1$ for all $d \geq 0$.
We see that $\a_2(2k) = 1$ and $\a_2(2k+1) = 2$ for all $k \geq 0$ since $G_2(d)$ is the path graph on $d+1$ vertices.
Theorem~\ref{thm:triangle} and Theorem~\ref{thm:4} affirm Conjecture~\ref{conj:1} for $n = 3$ and $n = 4$ respectively.
Howroyd's Conjecture~\ref{conj:AH} is exactly Conjecture~\ref{conj:period} for $n = 3$ with $N = 9$.

The sequence $\{\a_4(d)\}_{d \geq 0}$ begins
\[1, 4, 1, 80, 1, 944, 1, \dots\]
while the sequence $\{\a_5(d)\}_{d \geq 0}$ begins
\[1,5,1,705,5, \dots\]
which is as far as we could compute in each sequence.
Further computation for $n \geq 4$ would be valuable in testing both conjectures in this section.
Such computation seems difficult and could be interesting in its own right.

In order to prove a result similar to Theorem~\ref{thm:triangle} or Theorem~\ref{thm:4} for the next case of $n=5$ it would be very helpful to have a formula for $\alpha_5(d)$ similar to the known formula for both $\alpha_3(d)$ and $\alpha_4(d)$.
However, we are not aware of any such formula.
The values of the sequence  $\{\alpha_5(d)\}_{d \geq 0}$ are
\[1, 5, 7, 16, 26, \dots\]
and computation of these numbers quickly becomes difficult.

It is easy to compute $\a_n(d)$ for any $n \geq 1$ with $d \in \{0,1,2\}$.
The following gives the values of $a_n(0)$, $a_n(1)$, and $a_n(2)$ for any $n$.

\begin{proposition}
If $n \geq 0$, then $\a_n(0) = 1$, $\a_n(1) = n$, and $\a_n(2) = 1$.
\end{proposition}
\begin{proof}
We have that $\a_n(0) = 1$ since $G_n(0)$ is a single vertex.
Also, $\a_n(1) = n$ since $G_n(1)$ is the complete graph $K_n$.
For $n = 2$ we claim that $\a_n(2) = 1$ and that
\[\{x_i^2 : 1 \leq i \leq n\}\]
is the unique maximum independent set.
Indeed consider an independent set $I$ which contains the vertices
\[x_{i_1}x_{j_1}, x_{i_2}x_{j_2}, \dots, x_{i_r}x_{j_r}\]
for $i_s \neq j_s$.
It must be that $r \leq \tfrac{n}{2}$ since we are considering an independent set.
Furthermore, letting $A = \{i_s, j_s | 1 \leq s \leq r\}$ we must have $|A| = 2r$.
The independent set can then only contain $x_i^2$ for $i \not\in A$.
So, the size of $I$ can be at most $r + (n - 2r) = n - r$.

\end{proof}

\subsection{Domination}
We now discuss the relationship of our work with domination.
A \emph{$k$-dominating} set is a set vertices $D$ such that every vertex not in the set is adjacent to at least $k$ vertices in $D$.
A usual dominating set is then a $1$-dominating set.
Hopkins and Stanton~\cite{HS} characterized trees with a unique maximum independent set as well as graphs with a unique maximum independent set for which complement of this maximum independent set is also an independent set.
Further work in this direction was done by Siemes, Topp, and Volkmann~\cite{STV} by considering \emph{$k$-independent sets} for any $k$ which are independent sets $I$ such that any independent set $I'$ with $|I'| \geq |I| - (k-1)$ is must be a subset of $I$.
Hence, the notion of a unique maximum independent is recovered for $k=1$.
Certainly any maximal independent set will be a dominating set, but it is known that any $k$-independent set is actually a $(k+1)$-dominating set~\cite[Corollary 2]{STV}.
So, we obtain the following corollary.

\begin{corollary}
The unique maximal independent sets constructed in the proofs of Theorem~\ref{thm:triangle} and Theorem~\ref{thm:4} are $2$-dominating sets.
\end{corollary}

It can be seen from Figure~\ref{fig:G_3(12)} and Figure~\ref{fig:slices} that the independent sets from Theorem~\ref{thm:triangle} and Theorem~\ref{thm:4} are not $3$-dominating, and hence also they are not $2$-independent.
Volkmann~\cite{Volk} has studied further connections between unique maximum independent sets and $2$-dominating sets.

Recall, the complete bipartite graph $K_{1,r}$ has vertex set $\{a_1\} \cup \{b_i : 1 \leq i \leq r\}$ and edge set $\{\{a_1, b_i\} : 1 \leq i \leq r\}$.
A graph is $K_{1,r}$-free if it does not contain an induced subgraph isomorphic to $K_{1,r}$.

\begin{proposition}
The graph $G_n(d)$ is $K_{1,r}$-free if and only if $d < r$ or $n < r$.
\label{prop:K1r}
\end{proposition}
\begin{proof}
Assume that $d < r$ or $n < r$.
Consider a vertex $v_0$ which we may assume has $r$ neighbors $v_1, v_2, \dots v_r$.
Let 
\[v_i = \left(\frac{x_{i_1}}{x_{i_2}}\right)v_0\]
for some $i_1 \neq i_2$ such that $x_{i_2}$ divides $v_0$.
Since $d < r$ or $n < r$ there must be $v_i$ and $v_j$ with $i_2 = j_2$ but $i_1 \neq j_1$.
Then
\[\lcm(v_i, v_j) = \left(\frac{x_{i_1}x_{j_1}}{x_{i_2}} \right) v_0\]
which has degree $d+1$.
Hence $\{v_i, v_j\}$ is an edge.
It follows that in this case $G_n(d)$ must be $K_{1,r}$-free.

Now assume $d \geq r$ and $n \geq r$.
Take the vertex $v_0 = x_1^{d-r+1}x_2 \cdots x_r$.
Let
\[v_i = \left(\frac{x_{i}}{x_{i+1}}\right)v_0\]
for $1 \leq i \leq r-1$
and 
\[v_r = \left(\frac{x_{r}}{x_{1}}\right)v_0\]
Then we have that $G_n(d)|_X$ is an induced subgraph isomorphic to $K_{1,r}$ since
\begin{align*}
\lcm(v_0, v_i) &= x_i v_0 & \lcm(v_i, v_j) &= x_ix_jv_0
\end{align*}
for $1 \leq i, j \leq r$ with $i \neq j$.
Hence, in this case $G_n(d)$ is contains a copy of $K_{1,r}$ and the proposition is proven.
\end{proof}

The \emph{domination number} (i.e. minimal possible size of dominating set) and  \emph{independent domination number} (i.e. minimal possible size of dominating set which is also an independent set) of a graph $G$ is denoted by $\gamma(G)$ and $i(G)$ respectively.
We denote the domination number of $G_n(d)$ by $\gamma_n(d)$.
Similarly we let $i_n(d)$ denote the independent domination number of $G_n(d)$.
It is clear that $\gamma(G) \leq i(G)$ for any $G$.
Allan and Laskar proved that $i(G) = \gamma(G)$ when $G$ is $K_{1,3}$-free~\cite{K13}.
More generally, Bollob\'{a}s and Cockayne~\cite{indom} have shown that $i(G) \leq \gamma(G) (r-1) - (r-2)$ whenever $G$ is $K_{1,r+1}$-free.Hence, Proposition~\ref{prop:K1r} gives a bound between $i_n(d)$ and $\gamma_n(d)$.
However, we conjecture that more is true and that they in fact equal.

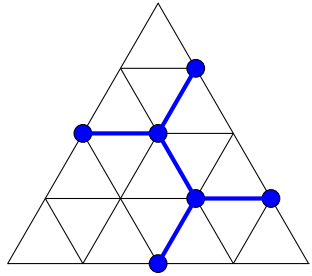
\begin{figure}
\centering
\begin{tikzpicture}
\begin{scope}[shift={(10,0)}]
\newcommand*\rows{4}
\foreach \row in {0, 1, ...,\rows} {
        \draw ($\row*(0.5, {0.5*sqrt(3)})$) -- ($(\rows,0)+\row*(-0.5, {0.5*sqrt(3)})$);
        \draw ($\row*(1, 0)$) -- ($(\rows/2,{\rows/2*sqrt(3)})+\row*(0.5,{-0.5*sqrt(3)})$);
        \draw ($\row*(1, 0)$) -- ($(0,0)+\row*(0.5,{0.5*sqrt(3)})$);
    }
   \node[draw, circle, fill=blue!, scale=0.65] at (2,0) {};
   \node[draw, circle, fill=blue!, scale=0.65] at (2.5,{0.5*sqrt(3)}) {};
   \node[draw, circle, fill=blue!, scale=0.65] at (3.5,{0.5*sqrt(3)}) {};
   \node[draw, circle, fill=blue!, scale=0.65] at (2,{sqrt(3)}) {};
   \node[draw, circle, fill=blue!, scale=0.65] at (1,{sqrt(3)}) {};
   \node[draw, circle, fill=blue!, scale=0.65] at (2.5,{1.5*sqrt(3)}) {};

\draw[ultra thick,blue] (2,0) -- (2.5,{0.5*sqrt(3)}) -- (3.5,{0.5*sqrt(3)});
\draw[ultra thick,blue] (2.5,{0.5*sqrt(3)}) --(2,{sqrt(3)});
\draw[ultra thick,blue] (2.5,{1.5*sqrt(3)}) --(2,{sqrt(3)}) -- (1,{sqrt(3)}) ;
\end{scope}
\end{tikzpicture}
\label{fig:notperfect}
\caption{An induced subgraph showing $G_3(4)$ is not domination perfect.}
\end{figure}

\begin{conjecture}
If $n \geq 1$ and $d \geq 0$, then $i_n(d) = \gamma_n(d)$.
\label{conj:igamma}
\end{conjecture}

It is easy to see that Conjecture~\ref{conj:igamma} holds for $n =1$ and $n=2$ with any $d$.
For $n=3$ it is conjectured by Wagon~\cite{chess} that 
\[\gamma_3(d) = \	\left\lfloor \frac{d^2+7d-23}{14} \right\rfloor\]
for any $d \geq 14$.
To our knowledge neither $\gamma_n(d)$ nor $i_n(d)$ has a proven formula for $n \geq 3$.
An upper bound of $\gamma_3(d)$ is provided in~\cite{chessBound}, but this upper bound is not tight.

A graph is called \emph{domination perfect} if $\gamma(H) = i(H)$ for every induced subgraph $H$.
This is a stronger property than what we have in Conjecture~\ref{conj:igamma}, and this strong property does not hold.
\begin{proposition}
For $n \geq 3$ and $d \geq 4$ the graph $G_n(d)$ is not domination perfect.
\end{proposition}
The proof of the proposition can be seen by considering $H$ to be the induced subgraph on
\[X = \{x_1^3x_3, x_1^2x_2^2, x_1^2x_2x_3,, x_1x_2x_3^2, x_1x_3^3, x_2^2x_3^2\}\]
which is shown in Figure~\ref{fig:notperfect} and has $\gamma(H) =2$ and $i(H) = 3$.
We note this induced subgraph $H$ is one of the 13 graphs forbidden graphs in Topp and Volkmann's sufficient condition for $i(G) = \gamma(G)$~\cite{equaldom}; hence, this result cannot be applied to prove the conjecture.

\bibliographystyle{alpha}
\bibliography{refs}

\end{document}